\documentclass{amsart}
\usepackage{amssymb,graphicx,rotating,multirow,multicol}
\usepackage{amsmath,amsthm,amsfonts,amscd}
\usepackage[toc,page]{appendix}
\usepackage{caption,comment}
\usepackage[mathscr]{euscript}
\usepackage{epstopdf}
\usepackage{float}

\usepackage{CJKutf8}
\usepackage{epsf}
\usepackage{epigraph}

\swapnumbers
\newtheorem{theorem}{Theorem}[section]
\newtheorem{lemma}[theorem]{Lemma}
\newtheorem{proposition}[theorem]{Proposition}

{\catcode`\@=11 \gdef\mnote#1{\marginpar{\footnotesize
 \tolerance\@M\spaceskip2.6\p@ plus10\p@ minus.9\p@\rm#1}}}
\def\Dg:{\endgraf{\bf Dg:\enspace}\ignorespaces}

\let\Bbb\mathbb

\def\sm{\smallsetminus}

\newcommand{\be}{\begin{equation}}
\newcommand{\ee}{\end{equation}}
\let\ge\geqslant 
\let\le\leqslant 

\let\til\widetilde
\let\e\varepsilon

\def\Z{\Bbb Z}
\def\R{\Bbb R}

\def\Rp#1{\Bbb{RP}^{#1}}

\let\+\sqcup

\makeatletter
\newcommand{\addresseshere}{%
  \enddoc@text\let\enddoc@text\relax
}
\makeatother

\begin{document}
\renewcommand{\arraystretch}{1.2}

%
%

\title[First homology of a real cubic is generated by lines]
{First homology of a real cubic is generated by lines}

\author[]
{S.~Finashin, V.~Kharlamov}
\address{Middle East Technical University,
Department of Mathematics\endgraf Ankara 06531 Turkey}
\address{
Universit\'{e} de Strasbourg et IRMA (CNRS)\endgraf 7 rue
Ren\'{e}-Descartes, 67084 Strasbourg Cedex, France}


\maketitle

\setlength\epigraphwidth{.57\textwidth}
\epigraph{
All things are difficult before they are easy.
}
{Thomas Fuller's (Gnomologia, 1732)
\\
Equivalent Ancient
Chinese Proverb:  
\begin{CJK*}{UTF8}{gbsn}
万事开头难.
\clearpage\end{CJK*}
}

\section{Introduction} The celebrated discovery by J.~Nash of existence of real algebraic structure on any smooth manifold has opened an entire
direction in real algebraic geometry where various
problems of modeling and approximating of other smooth objects by real algebraic ones were deeply investigated
(see the monograph \cite{bcr} for basic facts and general references).
Many of these problems,
such as asking
under which assumptions a vector bundle over smooth manifold can be realized algebraically
and a smooth map can be approximated, or realized up to homotopy, by an algebraic one,
have led in their turn to
various questions on representing homology classes by algebraic cycles.
 A simplest, but typical, example is the following statement: if all classes in $H^1(M;\Z/2)$,
where $M$ is a non-singular
real algebraic manifold, are algebraic, then each continuous map $M \to S^1$ can be approximated by
an entire rational map ({\it cf.} Lemma 14 in \cite{Ivanov}).

Recently, being motivated by investigation of the real integral Hodge conjecture for $1$-cycles \cite{hodge} and in connection
with a certain tight approximation property \cite{tight}, O.~Benoist and O.~Wittenberg
proved the following theorem.

\begin{theorem}{\rm (\cite[Theorem 9.23]{hodge})}\label{main} For each real non-singular cubic hypersurface $X$ of dimension $\ge 2$
the real lines on $X$ generate the whole group $H_1(X(\R);\Z/2)$.
\end{theorem}

The aim of our paper is to give a short proof of this statement that uses more
traditional and elementary tools.
We start from analysis of an appropriate similar statement for one-nodal hypersurfaces (see Proposition \ref{nodal}), and then
transfer the result to non-singular hypersurfaces by standard Morse theory arguments.
Concluding remarks contain some examples showing that Theorem \ref{main} cannot be generalized
in a straightforward way to (1) other rationally connected varieties like rational surfaces and hypersurfaces of degree 4 (of dimension $\ge 4$),
and (2) to higher homology groups of cubic hypersurfaces.

\subsection{Conventions} In what follows the homology groups are all with $\Z/2$-coefficients and we omit $\Z/2$ from notation.
Speaking on real algebraic varieties $X$, we identify $X$ with its complex point set and denote by $X(\R)$ its real locus.

\section{The case of nodal cubics}
Consider a cubic hypersurface $X_0\subset P^{n+1}$ with a singular point $p\in X_0$ of multiplicity two.
In affine coordinates centered at $p$ such a hypersurface
is defined by equation
$f_2+f_3=0$, where $f_2$ and $f_3$ are homogeneous polynomials
of degree 2 and 3, respectively.
These polynomials define on the infinity hyperplane $P^n\subset P^{n+1}$ a
quadric $Q$ with equation $f_2=0$ and a cubic $C$ with equation $f_3=0$.
Note that the quadric $Q$ and the intersection $Q\cap C$ are independent of the coordinate choice,
if the infinity hyperplane $P^n$ is identified with the space of lines in $P^{n+1}$ passing through $p$.
 Moreover, quadric $Q$ is naturally identified with
 the exceptional divisor $\til Q$ after blowing up $\til X_0\to X_0$ at point $p$, whereas $Q\cap C$ is identified
with the variety of lines on $X_0$ passing through $p$.

By definition, $p$ is a {\it node}
provided $Q$ is non-singular.

\begin{lemma}\label{transversal}
If $p$ is a node, then the intersection $Q\cap C$ is transversal if and only if
$X_0$ has no singular points other than $p$.
\end{lemma}

\begin{proof}
It is straightforward and well known (see, e.g., Lemma 2.2 in \cite{deformation}).
\end{proof}

By a {\it secant} of  $Q\cap C$ below we mean a line in $P^n$ intersecting  $Q\cap C$ at a pair of points, or at one point with multiplicity 2.
In terms of equations it means that $f_2$ and $f_3$ restricted to the line under consideration have two common zeros or have a common zero
of multiplicity 2 for $f_2$ and of multiplicity $\ge 2$ for $f_3$.

\begin{lemma}\label{lines-on-X0}
If $X_0$ is a cubic hypersurface with a node at a point $p$, then
the central projection from $p$ projects
a line on $X_0$ not passing through $p$ to a secant of $Q\cap C$.
Conversely, a
secant of $Q\cap C$ is a projection of a line on $X_0$ not passing through $p$ as soon as the secant is not contained in $Q$.
\end{lemma}

\begin{proof} To prove the first statement we choose affine coordinates so that $l$ is not contained in the infinity hyperplane and note
that vanishing of $f_2+f_3$ on $l$ implies that $f_3$ restricted to $l$ takes zero value at each of the zeros of $f_2$ with at least the same multiplicity.

Reciprocally, if a line $l'\subset P^n$ is a secant of $Q\cap C$, then we pick a linear parametrization $x_1=\phi_1(u,v), \dots, x_{n+1}=\phi_{n+1}(u,v)$
for $l'$, lift it up to a fractional parametrization $x_1=\phi_1(u,v)/\phi(u,v), \dots,  x_{n+1}=\phi_{n+1}(u,v)/\phi(u,v)$
of a line $l$ in $P^{n+1}$, and note that $l$ is contained in $X_0$ as soon as we pose $\phi = - f_3(\phi_1, \dots, \phi_{n+1})/f_2(\phi_1, \dots, \phi_{n+1})$,
where such
$\phi$ is linear in $u,v$ due to the definition of secants.
\end{proof}

If $X_0$ is {\it one-nodal} ({\it i.e.}, has a node at $p\in X_0$ and has no other singular points),
then blowing up of $P^{n+1}$ at $p$ gives a resolution $\til X_0\to X_0$ with an exceptional divisor $\til Q\subset\til X_0$,
while the central projection from $p$
induces a regular map $\pi :\til X_0\to P^n$ which is
isomorphic to blowing up of $P^n$ along $Q\cap C$.

\begin{lemma}\label{lines-on-X0-generate}
If $X_0$ is a real one-nodal cubic hypersurface, then
$H_1(\widetilde X_0(\R))$
is generated
by the proper images of real lines $l(\R)\subset X_0(\R)$.

 More precisely, it is sufficient to pick $b_0(Q(\R)\cap C(\R))+1$ real lines on $X_0$, namely,
for each connected component of $Q(\R)\cap C(\R)$ we take one arbitrary line
passing through the node and
projected to this component
and in addition we take one arbitrary line not passing through the node.
\end{lemma}

\begin{proof}
The blowup description of $\til X_0$ implies a splitting
\be
H_1(\widetilde X_0(\R))=H_1(P^2(\R))+H_0(Q(\R)\cap C(\R))
\label{decomposition}
\ee
and together with the description in Lemma \ref{lines-on-X0} it gives
 the required result
(for more details see Proposition 2.5 in \cite{deformation}).
\end{proof}



\begin{lemma}\label{cubic-surfaces}
For any real one-nodal cubic surface $X_0$ with the node $p\in X_0$ the group $H_1(X_0(\R)\sm p)$ is generated by
the real lines lying on $X_0$ and not passing through $p$.
\end{lemma}

\begin{proof}
The homeomorphism $X_0\sm p\cong \til X_0\sm\til Q$ and presentation of $\til X_0$ as $P^2$ blown up at six points of $Q\cap C$
gives 5 cases to be analyzed in the real setting. Namely, one case if
the number of imaginary pairs among these six points is $k=0,1,2$ and two cases for $k=3$ (one with $Q(\R)=\varnothing$ and another with  $Q(\R)\neq\varnothing$).

If $k=0$ or $1$,
splitting (\ref{decomposition}) implies that $H_1(\til X_0)$ is generated by
the class $l$ of  real lines $L$  in $P^2(\R)$ disjoint from
$Q(\R)\cap C(\R)$ and the classes  $e_i$
of the real exceptional curves $E_i$ ($i=1, \dots, 6-2k$) of the blowing up.
Then $\til Q(\R)$ represents the class $e_1+\dots+e_6$,
so that $H_1(\til X_0\sm\til Q)$, being its orthogonal complement,
is spanned by the classes $l+e_i+e_j$, $1\le i<j\le 6 -2k$
which are represented by
the proper images of real lines in $P^2(\R)$ connecting pairwise the real points of $Q\cap C$.

For $k=2$
and in the both cases of $k=3$, we use in addition the real line passing through a pair of imaginary points of $Q\cap C$.
\end{proof}

\begin{proposition}\label{nodal} For any
real one-nodal cubic hypersurface $X_0$ of dimension
$n \ge 2$
with the node $p\in X_0$, the group $H_1(X_0(\R)\sm p)$ is generated by
the real lines lying on $X_0$ and not passing through $p$.
\end{proposition}

\begin{proof}
The case $n=2$ was just analyzed, so, we assume that $n=\dim X_0\ge 3$
 and identify $H_1(X_0(\R)\sm p)$ with $H_1(\widetilde X_0(\R) \sm \widetilde Q(\R))$.

 If $\widetilde Q(\R)=\varnothing$, then $X_0(\R)\sm p$ is homeomorphic to $P^n(\R)$ and contains plenty of real lines (one per each of real planes through the node), so in this case the statement is trivial.

 Assume that $\widetilde Q(\R)\ne \varnothing$, and
consider the following segment of a long exact sequence for the pair $(\widetilde X_0(\R), \widetilde X_0(\R) \sm \widetilde Q(\R))$
\begin{eqnarray}\begin{split}
\ \ H_2(\widetilde X_0(\R), \widetilde X_0(\R) \sm &\widetilde Q(\R))\to H_1(\widetilde X_0(\R) \sm \widetilde Q(\R))\to H_1(\widetilde X_0(\R))
\\
&\to H_1(\widetilde X_0(\R), \widetilde X_0(\R)\sm \widetilde Q(\R)).
\end{split}\label{long}\end{eqnarray}
By Poincar\'e-Lefschetz duality, $H_1(\widetilde X_0(\R), \widetilde X_0(\R)\sm \widetilde Q(\R))=H^{n-1}(\widetilde Q(\R))=\Z/2$
which identifies the rightmost homomorphism with the intersection index homomorphism
$x\mapsto x\cap [\widetilde Q(\R)]$, while $H_2(\widetilde X_0(\R), \widetilde X_0(\R)\sm \widetilde Q(\R))=H^{n-2}(\widetilde Q(\R))=
H_1(\widetilde Q(\R))$
which identifies the leftmost homomorphism with taking the boundary of a small
band in $\widetilde X_0(\R)$
transversely intersecting   $\widetilde Q(\R)$ along its core circle.

This implies that $H_1(\widetilde X_0(\R) \sm \widetilde Q(\R))$ is generated by the image of
the leftmost homomorphism and any choice for a lift
to $H_1(\widetilde X_0(\R) \sm \widetilde Q(\R))$ of the sums
$[L^*_1(\R)]+[L^*_2(\R)]$ taken, in accord with Lemma \ref{lines-on-X0-generate},
over all pairs of real lines picked up
in different connected components
of the space of real lines passing through the node and contained in $X_0(\R)$
(symbol $*$ states for taking the proper image in $\widetilde X_0(\R)$).

As for
$[L^*_1(\R)]+[L^*_2(\R)]$ taken as above, the 2-plane generated by the two lines
 is not contained in $Q(\R)$.
Tracing a generic real 3-plane through $L_1$ and $L_2$, we get a real nodal cubic surface
$Y_0$.
Lemma \ref{cubic-surfaces} applied to
$Y_0$ shows that
the homology class $[L^*_1(\R)]+[L^*_2(\R)]$ can be lifted
to a class in $H_1(\widetilde Y_0(\R) \sm \widetilde Q(\R))$
represented by real lines contained in $\widetilde Y_0(\R) \sm \widetilde Q(\R)$, hence in
$\widetilde X_0(\R) \sm \widetilde Q(\R)$.

As for the image of the leftmost homomorphism of (\ref{long}),
it is generated by the boundaries of
thin bands in $\til X_0(\R)$ transversally crossing $\til Q(\R)$ along real lines
$l^*(\R)$,
because the first homology group
of $\widetilde Q(\R)= Q(\R)$, as for that of any real quadric,
is known to be generated by real lines.
The boundary 1-cycle of such a band
is clearly bordant to the union of two real lines in $P^n(\R)$: one inside $Q(\R)$ and another outside.
Let us pick one of these two lines, $h(\R)$,
and consider a generic real 3-plane $H\subset P^{n+1}$ passing through $h$ and the node.
Then, $H(\R)\cap X_0(\R)$
is a real nodal cubic surface, where
the proper image $h^*(\R)$ of $h(\R)$ does not pass through the node and, hence, by Lemma \ref{cubic-surfaces}
its homology class is spanned by the classes of real lines
contained in the surface $H(\R)\cap X_0(\R)$
and not passing through the node.
\end{proof}

\section{Passing to non-singular cubics}
Any non-singular real cubic $n$-dimensional hypersurface can be included into a smooth family $\{X_t\}$, $t\in\R$, of real cubics intersecting
transversely the discriminant hypersurface
in the projective space formed by all cubics  $X\subset P^{n+1}$. We call such $\{X_t\}$ a {\it Morse family}, since variation
of $t$ gives Morse modifications as $t$ passes values $t_0$ at which $X_{t_0}$ is singular (one-nodal, because of our transversality condition).

Recall that the {\it Morse index} of such a modification
is equal to $i$ if for $t<t_0$
close to $t_0$
the vanishing real sphere $S_t\subset X_t$ is of dimension $i-1$, $0\le i\le n+1$,
(then for $t>t_0$ close to $t_0$ it has dimension $n-i$).
For instance, in the case $i=0$, for $t<t_0$ the vanishing real sphere is empty,
and so, the Morse modification
 is  a ``birth'' of spherical component $S^n$ in $X_t$, $t>t_0$, while for index $i=n+1$
Morse modification is a ``death'' of $S^n$ in $X_t$, $t<t_0$.

As is well known, the family $\{X_t\sm S_t\}_{\vert t-t_0\vert <\epsilon}$ (where $S_0$ stands for $p$) forms a trivial fibration
over $\vert t-t_0\vert <\epsilon $, and therefore the exists a well defined natural {\it parallel transport} isomorphism $H_1(X_0(\R)\sm p)\to H_1(X_t(\R)\sm S_t)$
for each $t$ in $\vert t-t_0\vert <\epsilon$.

\begin{lemma}\label{for2andhigher} If $X_t$ is
an index $i\ne 1$ Morse family of dimension $n\ge 2$
real cubic hypersurfaces
and $S_t$, $t>0$, a continuous family of vanishing
spheres collapsing to the node $S_0=p$ of $X_0$, then
the composition of the parallel transport
isomorphism $H_1(X_0(\R)\sm p)\to H_1(X_t(\R)\sm S_t)$, $t>0$,
 with the inclusion homomorphism $H_1(X_t(\R)\sm S_t)\to H_1(X_t(\R))$  is an epimorphism.
\end{lemma}

\begin{proof} It is a straightforward consequence of the long exact sequence
$$
H_1(X_t(\R)\sm S_t)\to H_1(X_t(\R))\to H_1(X_t(\R),X_t(\R)\sm S_t)
 $$
and the Poincar\`e-Lefschetz duality
$$H_1(X_t(\R);X_t(\R)\sm S_t) \cong H^{n-1}(S_t)=0,
$$
since $\dim S_t=n-i$ with $i\ne 1$ and $n\ge 2$.
\end{proof}

\begin{lemma}\label{stability}
For any Morse family $X_t$ of real cubics
with $-\e<t<\e$
the parallel transport homomorphism $H_1(X_0(\R)\sm p)\to H_1(X_t(\R)\sm S_t)$ maps classes of real lines
not passing through $p$ to classes of real lines not intersecting $S_t$.
\end{lemma}

\begin{proof} According to Lemma 7.7 in \cite{CC} the variety of real lines not passing through $p$ is smooth and of pure dimension $2(n-2)$. Furthermore, by Corollary 7.6 in \cite{CC} the balanced lines (lines of type I in terminology of \cite{CC}) form there an open dense subset. Thus, there remains to notice that the balanced lines are stable under deformation, since they can be seen as transversal zeros of the section defined by the equations of the cubics in the family in the corresponding vector bundle over the Grassmannian of lines.
\end{proof}

\begin{proof}[Proof of Theorem \ref{main}]
In dimension $2$ the claim of the theorem is well known. So, we assume that $\dim X\ge 3$ and
include $X$ as $X_t$ with $t>0$ in a Morse family of real cubics perturbing
a nodal one, $X_0$.
If its Morse index $\ne 1$ then the claim
follows from Lemmas \ref{for2andhigher} and \ref{nodal}.

In the case of Morse index $1$ we apply Lemmas \ref{nodal} and
\ref{stability},
and notice in addition that each of real lines in $X_0(\R)$ passing through the node
can be varied continuously as a line in $X_t$ for $t>0$ (see lemma \ref{index1-perturbation} below).
\end{proof}

\begin{lemma}\label{index1-perturbation}
Consider a Morse family $X_t$, $-\e<t<\e$, of real cubic hypersurfaces in $P^{n+1}$, $n\ge2$, perturbing
a one-nodal cubic $X_0$ with index $1$. Then any real line $l_0\subset X_0$, $p\in l_0$ can be extended to
a continuous family of real lines $l_t\subset X_t$, for $0\le t\le\e$.
\end{lemma}

\begin{proof} For $n=2$ (case of cubic surfaces), this fact is well known and simple.
In dimensions $n\ge3$, we take a generic real 3-plane $H\subset P^{n+1}$ containing line $l_0$
and note that $X_t\cap H$ is a Morse family of the same index 1. Then we obtain a required family $l_t\subset X_t\cap H$.
\end{proof}

\section{Concluding Remarks}

\subsection{Real rational surface whose group $H_1(X(\R))$ is not generated by embedded rational curves}
The desired
examples are provided
by maximal real del Pezzo surfaces $X$ of degree $K_X^2=1$.
Recall that such a surface
has the real locus $X(\R)=\Rp2\+4S^2$
and
the complex conjugation
involution in
$H_2(X; \Z)$
acts as the reflection in the hyperplane orthogonal to $K_X$. By this reason, the only classes of real curves are $-mK_X$ with $m\ge 1$.
By the adjunction formula, the arithmetic genus of these curves is
$g=1+\frac12(m^2-m)\ge1$ hence $X$ contains no embedded real curves of genus $0$.

\subsection{Real quartics without line generation of $H_1$}
The locus of complex points of a non-singular projective quartic hypersurface
of dimension $4+m\ge 4$ is pointwise covered by complex lines.
However, it happens to be insufficient
for generating the first homology of the real locus by real lines.
Here is an example which
can be obtained by a small variation of a nodal quartic. Namely, it is sufficient to pick as a starting point
a nodal hypersurface $X_0$ given by affine equation $f(x_0,\dots, x_{4+m})=0$
%
$$
f(x)=f_2(x)+\e f_4(x),\ f_2(x)=\sum_{i=1}^{4+m}x_i^2-x_0^2,\
 f_4(x)=(\sum_{i=1}^{4+m}x_i^2-2x_0^2) (\sum_{i=1}^{4+m}x_i^2+x_0^2).
$$
Such a hypersurface contains no real lines passing through the node (since $f_2=f_4=0$ has no real roots except $0$).
But, if $\e$ is sufficiently small, it contains loops passing through the node and homologous to real lines
in the ambient projective space
(since then $X_0(\R)$ is topologically equivalent to the real cone defined by $f_2$).
Moreover, for small $\e >0$ the Morse modification $X_t$ of $X_0$ given by equation $f(x)=t$ with $t>0$ close to $0$
replaces the nodal point by $S^{3+m}$ not homologous to $0$,
and in $X_t(\R)$, due to absence of real lines through the node in $X_0$, the homology class Poincar\'e dual to $[S^{3+m}]$ can not be realized by a combination of real lines.

As it was communicated to us by O.~Wittenberg, there exist other examples, suggested by J.~Koll\'ar and F.~Mangolte \cite{KM}  and based on a different idea
that provides examples with even stronger properties.
Namely, it is sufficient to pick $m+4\ge 3$ quadratic forms
$q_1, ..., q_{m+4}$ in $m+6$ variables such that the system $q_1=...=q_{m+4}=0$ defines a
smooth curve
in $P^{m+5}$ whose real locus has at least two connected
components.  Then, one can prove by an appropriate bounded degrees limit argument ({\it cf.} \cite[Example 26.2]{KM} and \cite[Example 9.10]{hodge}) that, for any $d\ge 1$,
the first homology group $H_1(X_t(\R),\Z/2\Z)$
of the quartic hypersurfaces $X_t$ defined
in $P^{m+5}$ by
equation
$$
q_1^2 + ... + q_{m+4}^2 = t(x_0^4 + ... + x_{m+5}^4)
$$
is not spanned by the fundamental classes of algebraic curves of degree
$\le d$,  as soon as $t>0$ is small enough.

As we show in the final remark, our examples open a way to generalizations in another direction,
to higher dimensional homology groups.

\subsection{Real cubics without 2-plane generation of $H_2$}
Note first, that each real line contained in the projective quadric $\sum^{3+m}_{i=0} y_i^2 - u^2- v^2= 0$ can be generated by 2 real points
$(A,1,0)$ and $(B,0,1)$ with $\sum_{i=0}^{3+m}A_i^2=1=\sum_{i=0}^{3+m} B_i^2, \sum_{i=0}^{3+m} A_iB_i=0$.
This implies that, for any $\e_1\ne 0$
the quadro-cubic given by equations
$$
\sum^{3+m}_{i=0} y_i^2 - u^2- v^2= 0,\
y_0^3+\e_1(u^3+v^3)=0
$$
contains no real lines at all, but for $\e_1$ close to 0 it contains real pseudo-lines (since the real locus of such a quadro-cubic is isotopic to
that of the hyperplane section $\sum^{3+m}_0 y_i^2 - u^2- v^2= 0,
y_0=0 $).
 Hence, for any $\e_2\ne 0$ sufficiently small,
the nodal real cubic of dimension $5+m\ge 5$ given by affine equation
$$
\sum^{3+m}_{i=0} y_i^2 - u^2- v^2= \epsilon_2
(y_0^3+\epsilon_1(u^3+v^3))
$$
contains no real 2-plane passing through the node,
but
contains real 2-pseudo-planes passing through the node. The latter implies
that the Morse modification replacing the node by $S^{3+m}$ leads to $S^{3+m}$ not homologous to zero (to check it, note that the link of the node is naturally identified, on one hand, with the boundary of tubular neighborhoods of $S^{3+m}$ and, on the other hand, with the lift of the real locus of the quadric $\sum^{3+m}_{i=0} y_i^2 - u^2- v^2=0$
(in $\R P^{5+m}$)
defining the node to $S^{5+m}$ with respect to the standard covering $S^{5+m}\to \R P^{5+m}$), and by this reason
the $H_2$ of the non-singular real cubic obtained is not generated by real 2-planes.

\end{document}